\documentclass[11pt]{article}
\usepackage[english]{babel}
\usepackage{amsmath}
\usepackage{amsthm}
\usepackage{amssymb}
\usepackage{color}
\newtheorem{theorem}{Theorem}[section]
\newtheorem{definition}[theorem]{Definition}
\newtheorem{lemma}[theorem]{Lemma}
\newtheorem{proposition}[theorem]{Proposition}

\newtheorem{remark}[theorem]{Remark}

\newcommand{\s}{\mathbb{S}}

\newcommand{\rr}{\mathbb{R}}

\newcommand{\hh}{\mathbb{H}}

\title{\bf Extension results for slice regular functions of a quaternionic variable}

\author{Fabrizio Colombo\\ \normalsize Dipartimento di Matematica, Politecnico di
Milano\\ \normalsize Via Bonardi, 9, 20133 Milano, Italy,
fabrizio.colombo@polimi.it \and Graziano Gentili\\
\normalsize Dipartimento di Matematica, Universit\`a di Firenze, \\
\normalsize Viale Morgagni, 67 A, Firenze, Italy,
gentili@math.unifi.it
\\
\and Irene Sabadini\\ \normalsize Dipartimento di Matematica,
Politecnico di Milano\\ \normalsize Via Bonardi, 9, 20133 Milano,
Italy, irene.sabadini@polimi.it \and Daniele Struppa
\\ \normalsize Department of Mathematics and Computer Science
\\ \normalsize Schmid college of science
\\ \normalsize Chapman University, Orange, CA 92866 USA,
struppa@chapman.edu }

\date{ }
\begin{document}
\maketitle
\begin{abstract}
In this paper we prove a new representation formula for slice regular functions, which shows that the value of a slice regular function $f$ at a point $q=x+yI$ can be recovered by the values of $f$ at the points $q+yJ$ and $q+yK$ for any choice of imaginary units $I, J, K.$ This result allows us to extend the known properties of slice regular functions defined on balls centered on the real axis to a much larger class of domains, called axially symmetric domains. We show, in particular, that axially symmetric domains play, for slice regular functions, the role played by domains of holomorphy for holomorphic functions.
\end{abstract}


\section{Introduction}
The theory of slice regular functions has been recently introduced in \cite{gs} and \cite{advances}. Since then, the theory has been extensively studied in a series of papers which show on one hand the richness of this class of functions (see for example \cite{cgs}, \cite{caterina} and \cite{open}).  At the same time, this theory has demonstrated its interest by allowing a new application to the theory of quaternionic linear operators (see \cite{cgssann}, \cite{cgss}, \cite{cgss2}, \cite{cosa}).  The theory of slice regular functions is quite different from the more classical theory of regular functions in the sense of Cauchy-Fueter, see \cite{fueter 1}, \cite{fueter 2}, and the more recent \cite{csss}, and when compared with such theory, it shows many new and interesting features, such as the fact that it includes both polynomials and power series in the variable $q$.

 The original presentation of the theory, and the results obtained so far, have however a fundamental shortcoming. Up to now, indeed, the entire theory relies on the fact that slice regular functions admit power series expansion only on balls with center at real points. This clearly limits the applicability of the theory. In a recent paper \cite{cgs}, however, we proved a formula which relates the value that a slice regular function assumes at a point $q$ to the values it assumes on two suitable conjugate points. This formula, interesting in its own merit, can actually be generalized by allowing replacing the two conjugate points with two points who lie on two different complex planes (Theorem \ref{GRF}). It is such a generalization that turns out to be crucial  to prove the Extension Theorem \ref{circular extendability}, which will allow to extend the validity of the results known for quaternionic power series to the case of slice regular functions defined on suitable domains  called axially symmetric s-domains (see Definition \ref{def_circular}). We will show that axially symmetric s-domains have the same features, for slice regular functions, of domains of holomorphy for holomorphic functions.
In particular, this permits us to extend some of the main results in \cite{advances}, and to obtain a notion of product of slice regular functions which gives a slice regular result, as well as a notion of regular reciprocal of a function.

\section{Preliminaries on slice regular functions}

In this section we collect the basic definitions, and the basic results already obtained, in the theory of slice regular functions.

We recall that the skew field of quaternions $\hh$ is obtained by endowing $\rr^4$ with the classical multiplication operation: if $1,i,j,k$ denotes the standard basis, define $i^2 = j^2 = k^2 = -1,$  $ ij = -ji = k, jk = -kj = i, ki = -ik = j,$ let $1$ be the neutral element and extend the operation by distributivity to all quaternions $q = x_0 + x_1 i + x_2 j + x_3 k$. If we define the \emph{conjugate} of such a $q$ as $\bar q = x_0 - x_1 i - x_2 j - x_3 k$, its \emph{real} and \emph{imaginary part} as $Re(q) = x_0$ and $Im(q) = x_1 i + x_2 j + x_3 k$ and its \emph{modulus} by $|q|^2= q\bar q = Re(q)^2 + |Im(q)|^2$, then the multiplicative inverse (the reciprocal) of each $q \neq 0$ is computed as $q^{-1} = \frac{\bar q}{|q|^2}.$ In the sequel we will denote by $\mathbb{S}$
the unit 2-sphere  of purely imaginary
quaternions, i.e.
$$
\mathbb{S}=\{q\in \hh : Re(q) = 0 \ \ {\rm and}\ \  |Im(q)|=1 \}.
$$
Any element $I\in\mathbb{S}$ is characterized by the equation
$I^2=-1$ and for this reason the elements of $\mathbb{S}$ are called
imaginary units.
Moreover, for any $I\in\mathbb{S}$, one may consider the complex plane $L_I={\rm span}_{\mathbb{R}}\{1,I\}$ whose elements will be denoted by $x+yI$. Note that to any non-real quaternion $q$ is associated a unique element in $\mathbb{S}$ by the map
$$
q\mapsto\frac{{\rm Im}(q)}{|{\rm Im}(q)|}:=I_q.
$$
Thus, for any non-real
quaternion $q\in \mathbb{H} \backslash \mathbb{R}$, there exist,
and are unique, $x, y \in \mathbb{R}$ with $y>0$ and $I_q\in \mathbb{S}$, such that $q=x+yI_q$.
If the quaternion $q$ is real then, since $y=0$, $I_q$ can be any element of $\mathbb{S}$.

Let us now recall the definition of slice regularity (rephrased here in a slightly more general way than in the
original \cite{advances}):
\begin{definition}\label{regularity} Let $\Omega$ be a domain in
$\mathbb{H}$. A function $f:\Omega \to \mathbb{H}$ is said to be
slice left regular if, for every $I \in
\mathbb{S}$, its restriction $f_I$ to the complex line
$L_I=\mathbb{R}+\mathbb{R}I$ passing through the origin and
containing $1$ and $I$ has continuous partial derivatives and satisfies
$$\overline{\partial}_If(x+yI):=\frac{1}{2}\left(\frac{\partial}{\partial x}
+I\frac{\partial}{\partial y}\right)f_I(x+yI)=0,$$ on $\Omega \cap
L_I$.
\end{definition}
\noindent One may give an analogous definition for slice right regularity. This latter notion will give a theory completely equivalent to the one of slice left regular functions.
In this paper,
unless explicitly stated, we will always
refer to slice left regular functions and call them \emph{regular functions} for short. It is worth to point out that an analogous theory can be developed for functions with values in a Clifford algebra which are called slice monogenic, see \cite{slicecss}.


A notion of ``Gateaux-type" derivative can be given for regular functions, as follows:

\begin{definition}\label{derivative}
Let $\Omega$ be a domain in $\mathbb{H}$, and let $f:\Omega \to
\mathbb{H}$ be a regular function. The slice derivative (in
short derivative) of $f$, $\partial_s f$, is defined as follows:
\begin{displaymath}
\partial_s(f)(x+yI) = \partial_I(f)(x+yI)= \frac{1}{2}\left(\frac{\partial}{\partial x}
-I\frac{\partial}{\partial y}\right)f_I(x+yI)
\end{displaymath}
\end{definition}
Notice that the definition of derivative is well posed because
it is applied only to regular functions for which
$$
\frac{\partial}{\partial x}f(x+yI)= -I\frac{\partial}{\partial
y}f(x+yI)\qquad \forall I\in\mathbb{S},
$$
and therefore, analogously to what happens in the complex case,
$$ \partial_s(f)(x+yI) =
\partial_I(f)(x+yI)=\partial_x(f)(x+yI).
$$
Indeed, as a consequence, for the quaternions $x$ of the real axis, for which a corresponding $I_x$ is not uniquely defined, we obtain, independently of $I_x$,
$$
\partial_s(f)(x)=\partial_x(f)(x).
$$
If $f$ is a regular function, then also its derivative
is regular because we have
\begin{equation}\label{derivatareg}\overline{\partial}_I(\partial_sf(x+Iy))
=\partial_s(\overline{\partial}_If(x+Iy))=0,\end{equation} and
therefore
$$
\partial^n_sf(x+yI)=\frac{\partial^n
f}{\partial x^n}(x+yI).
$$
A crucial result which we will widely use in this paper is the following (see \cite{advances}):
\begin{lemma}[Splitting Lemma]\label{lemma spezzamento}  If $f$ is a regular function
on an open set $\Omega$, then for
every $I \in \mathbb{S}$, and every $J$ in $\mathbb{S}$,
perpendicular to $I$, there are two holomorphic functions
$F,G:\Omega\cap L_I \to L_I$ such that for any $z=x+yI$, it is
$$f_I(z)=F(z)+G(z)J.$$
\end{lemma}
\begin{remark}{\rm The statement of the Splitting Lemma given in \cite{advances} is given for functions defined on open balls with center at the origin, but the proof works for any open set.}\end{remark}
Another fundamental property of regular functions is that they admit power series expansion when defined on balls with center at real points (see \cite{advances}):
\begin{theorem}\label{svilluppo}
Let $B$ be a ball with center at a real point $p_0$ and with radius $R>0$.
A function $f:B \to \mathbb{H}$ is regular if, and only if,  it has a series
expansion of the form
$$f(q)=\sum_{n=0}^\infty (q-p_0)^n \frac{1}{n!} \frac{\partial^n f}{\partial x^n}(p_0)$$
converging on B.
\end{theorem}
Even if the definition of regular functions can be given on open sets of $\hh$, the corresponding function theory is meaningful only when we consider  special classes of open sets. Consider, for instance,  the domain
$$
\Omega =\mathbb{H}\setminus \mathbb{R}
$$
It is immediate to see that every function of the form
$$
f(q)=\left\{
\begin{array}{lll}
&1\quad &\quad{\rm on}\  \Omega \cap L_I, I\not= i \\
&q^2=(x+iy)^2&\quad{\rm on}\  \Omega \cap L_i \\
\end{array}\right.
$$
is regular but it is not even continuous. A class of domains which allows to avoid such
undesirable phenomena is introduced in the following

\begin{definition}
Let $\Omega \subseteq \mathbb{H}$ be a domain in $\mathbb{H}$. We
say that $\Omega$ is a \textnormal{slice domain (s-domain for
short)} if $\Omega \cap \mathbb{R}$ is non empty and if $L_I\cap
\Omega$ is a domain in $L_I$ for all $I \in \mathbb{S}$.
\end{definition}
Notice that the requirement that an s-domain $\Omega$ be connected
on every $L_I$ is stronger than requiring that $\Omega$ be
connected. However, this hypothesis turns out to be
necessary since, by definition, regular functions are
holomorphic on every slice $L_I$, for all $I\in \mathbb{S}$.
On s-domain we can extend the validity of Theorem \ref{svilluppo}:
\begin{proposition}\label{serie_domini}
Let $f$ be a regular function on an s-domain $\Omega$. Then for any real point
$p_0$ in $\Omega$, the function
$f$ can be represented in power series
$$f(q)=\sum_{n=0}^\infty (q-p_0)^n \frac{1}{n!} \frac{\partial^n f}{\partial x^n}(p_0)$$
on the ball $B(p_0,R)$ where $R=R_{p_0}$ is the largest positive
real number such that $B(p_0,R)$ is contained in $\Omega$.
\end{proposition}
\begin{proof}
Since $f$ is regular in $p_0$ then, for every $I\in\mathbb{S}$,
$f$ can be expanded in power series on the maximal disc $\Delta_I(p_0,a_I)$
of radius $a_I$ in
$\Omega\cap L_I$. The radius $R$ turns out to be  $\min_{I\in\mathbb{S}} a_I$
which is nonzero because $p_0$ is an internal point in $\Omega$.
\end{proof}
Another important result whose validity holds on s-domain is the Identity Principle, which can be stated (and proved as in \cite{advances}) as follows:
\begin{theorem}[Identity Principle] \label{identity principle} Let $f:\Omega\to\mathbb{H}$ be a regular function
on an s-domain $\Omega$. Denote by $Z_f=\{q\in \Omega : f(q)=0\}$
the zero set of $f$. If there exists $I \in \mathbb{S}$ such that
$L_I \cap Z_f$ has an accumulation point, then $f\equiv 0$ on
$\Omega$.
\end{theorem}
Among the s-domains, we will be interested in a special subclass identified by the additional requirement that they are axially symmetric, according to the following:
\begin{definition}\label{def_circular}
Let $\Omega \subseteq \mathbb{H}$. We say that $\Omega$ is
\textnormal{axially symmetric} if, for all $x+yI \in \Omega$, the whole
2-sphere $x+y\mathbb{S}$ is contained in $\Omega$.
\end{definition}
As we will see in the sequel of this paper, the axially symmetric s-domains will turn out to be the natural domains of definition for regular functions, much like the domains of holomorphy are the natural domains for holomorphic functions.

\section{Representation Formulas}
This section contains two different  Representation Formulas, at different levels of generality, for regular functions defined on axially symmetric s-domains. These formulas will then play a key role in the next sections, where they will allow us to prove the Extension Theorems for regular functions.

The next result shows that the value of a regular function at a point $q$ belonging to an axially symmetric s-domain is related to the values that the function assumes at two conjugate points on a same plane $L_J$ and belonging to the 2-sphere associated to $q$ (see Theorem 2.26 in \cite{cgs}, as well as \cite{cosa}):
\begin{theorem}[Representation Formula]\label{formula} Let
$f$ be a regular function on an axially symmetric s-domain
$\Omega\subseteq  \mathbb{H}$. Choose any $J\in \mathbb{S}$.  Then
the following equality holds for all $q=x+yI \in \Omega$:
\begin{equation}\label{distribution}
f(x+yI) =\frac{1}{2}\Big[   f(x+yJ)+f(x-yJ)\Big]
+I\frac{1}{2}\Big[ J[f(x-yJ)-f(x+yJ)]\Big].
\end{equation}
Moreover, for all $x, y \in \mathbb{R}$ such that $x+y\mathbb{S}
\subseteq \Omega$, there exist $b, c \in \mathbb{H}$ such that for
all $K \in \mathbb{S}$ we have
\begin{equation}\label{cappaI}
\frac{1}{2}\Big[   f(x+yK)+f(x-yK)\Big]=b \quad \quad {\sl and}
\quad \quad \frac{1}{2}\Big[ K[f(x-yK)-f(x+yK)]\Big]=c.
\end{equation}
As a consequence $b$ and $c$ do not depend on $K\in\mathbb{S}$.
\end{theorem}
The following generalization of the formula (\ref{distribution}) will be essential to prove the Extension Theorem, which is one of the main results of this paper.

\begin{theorem}[General Representation Formula]\label{GRF}
Let $f$ be a regular function on an axially symmetric s-domain $\Omega\subseteq
\mathbb{H}$. For any choice of $J,K\in \mathbb{S}$, $J\not= K$, the following equality holds for all $q=x+yI \in \Omega$:
\begin{equation}\label{distributionJK}
f(x+yI)
=(J-K)^{-1}[  J f(x+yJ)-Kf(x+yK)]
+I(J-K)^{-1}[ f(x+yJ)-f(x+yK)]
\end{equation}
Moreover, for all $x, y \in \mathbb{R}$ such that $x+y\mathbb{S} \subseteq \Omega$, there exist $b, c \in \mathbb{H}$ such that for all  $J,K\in \mathbb{S}$, $J\not= K$ we have
\begin{eqnarray}\label{cappaJK}
(J-K)^{-1}\Big[  J f(x+yJ)-Kf(x+yK)\Big]=b \\
\quad \quad(J-K)^{-1}\Big[ f(x+yJ)-f(x+yK)]\Big]=c.
\end{eqnarray}
\end{theorem}
\begin{proof} If $q$ is real the proof is immediate. Otherwise, for all $q=x+yI$, we define the function
\begin{eqnarray*}\label{cappaJK1}
\phi(x+yI) =(J-K)^{-1}\Big[  J f(x+yJ)-Kf(x+yK)\Big]
+I(J-K)^{-1}\Big[ f(x+yJ)-f(x+yK)\Big]
\\
= [(J-K)^{-1}J+I(J-K)^{-1}] f(x+yJ) - [(J-K)^{-1}K + I(J-K)^{-1}]f(x+yK).
\end{eqnarray*}
As we said, for all  $q \in \Omega\cap \mathbb{R}$  we have
$$
\phi(q)=\phi(x)=f(x)=f(q).
$$
Therefore if we prove that $\phi$ is slice regular on $\Omega$,
the first part of the assertion will follow from the identity principle
for slice regular functions.
Indeed, since $f$ is slice regular on $\Omega$, for any $H\in \mathbb{S}$ we have $\frac{\partial}{\partial x}f(x+yH)= -H\frac{\partial}{\partial y}f(x+yH)$
on $\Omega\cap L_{H}$; hence
$$
\frac{\partial \phi}{\partial x}(x+yI)=-[(J-K)^{-1}J+I(J-K)^{-1}]J\frac{\partial f}{\partial y} (x+yJ) + [(J-K)^{-1}K + I(J-K)^{-1}]K\frac{\partial f}{\partial y}(x+yK)
$$
and also
$$
I\frac{\partial \phi}{\partial y}(x+yI) = I [(J-K)^{-1}J+I(J-K)^{-1}] \frac{\partial f}{\partial y}(x+yJ) - I[(J-K)^{-1}K + I(J-K)^{-1}]\frac{\partial f}{\partial y}(x+yK).$$
It is at this point immediate that
$$
\Big(\frac{\partial}{\partial x}+I\frac{\partial}{\partial
y}\Big)\phi(x+yI)=0
$$
and equality (\ref{distributionJK}) is proved. The formulas (\ref{cappaJK}) and (\ref{cappaJK1}) immediately
follow from (\ref{cappaI}).
\end{proof}

We will now prove a nice geometric property for a class of domains of $\mathbb{H}$, which includes the s-domains. To this aim we first need the following

\begin{definition}\label{circolarizzato}
If $S $ is a subset of $\mathbb{H}$, then the set \begin{equation}\label{formula circolarizzato}
\widetilde{S}=\bigcup_{x+yJ\in S}(x+y\mathbb{S})
\end{equation}
is called the \emph{(axially) symmetric completion} of $S$.
\end{definition}

\begin{proposition}\label{circolarizzazionedidomini}
Let $\Omega \subseteq \mathbb{H}$ be a domain such that $\Omega\cap \mathbb{R} \neq \emptyset$. Then $\widetilde{\Omega}$ is an axially symmetric s-domain.
\end{proposition}
\begin{proof} Since $\Omega$, and hence $\widetilde \Omega$,  contains some real points, we only need to prove that, for all $I\in \mathbb{S}$, the intersection $\widetilde \Omega\cap L_I$ is  connected in $L_I$. Indeed, we will fix a real point $x\in \Omega \subseteq \widetilde \Omega$ and, for any point  $\tilde p \in \widetilde \Omega\cap L_I$, we will construct an arc in $\widetilde \Omega \cap L_I$ with endpoints $\tilde p$ and $x$ . In fact, by definition, there exist $p=Re(p)+Im(p)$  in $\Omega$ such that $\tilde p=Re(p)\pm |Im(p)|I$. Since $\Omega$ is (connected and open, and then) arcwise connected, there exists an arc $\gamma(t)=Re(\gamma(t))+Im(\gamma(t))$ connecting $p$ and $x$ in $\Omega$.  Then the arc $\gamma_I^+(t)=Re(\gamma(t))+|Im(\gamma(t))|I$ or the arc $\gamma_I^-(t)=Re(\gamma(t))-|Im(\gamma(t))|I$ connects $\tilde p$ to $x$ in $\widetilde \Omega\cap L_I$. This concludes the proof.
\end{proof}

\begin{remark}{\rm
In the statement of the last proposition, the hypothesis $\Omega\cap \mathbb{R} \neq \emptyset$ cannot be removed. In fact  $D=\{ (x+yI) : x^2+(y-2)^2 <1, I\in \mathbb{S}\}$ is an example of an axially symmetric domain of $\mathbb{H}$ that is not an s-domain. Even if we require that for any $I\in \mathbb{S}$ the intersection $\Omega \cap L_I$ is connected, the hypothesis $\Omega\cap \mathbb{R} \neq \emptyset$ remains essential. To see this, consider the hemisphere $\mathbb{S}^+=\{ x_1i+x_2j+x_3k \in \mathbb{H} :  x_1^2+x_2^2+x_3^2=1, x_3>0 \}$. The domain $E=\{ (x+yI) : x^2+(y-2)^2 <1, I\in \mathbb{S}^+\}\subset  \mathbb{H}$ is such that $E\cap L_I$ is connected for all $I\in \mathbb{S}$ but its symmetric completion $\widetilde E=\{ (x+yI) : x^2+(y-2)^2 <1, I \in \mathbb{S}\}$ is not an s-domain.}
\end{remark}

\section{Extension results}

We are going to prove a result which allows to extend any slice regular function defined on an s-domain $\Omega$ to its symmetric completion.
In this sense, we obtain that the axially symmetric s-domains are the analog of the domains of holomorphy in the complex setting. We also prove Theorem \ref{gef} dealing with the more general situation  in which $\Omega$ is not open.
In view of this result, we can extend a ``holomorphic" map $f:\Omega_D\subseteq \mathbb{C}\to \mathbb{H}$, defined on a domain symmetric with respect to the real axis,
to a regular function defined on the symmetric completion of $\Omega_D$.
These Extension Theorems and the Representation Formulas proved in the previous section will allow us to extend the validity of the results obtained in \cite{advances} from the setting of regular functions defined on balls with center at real points (and thus admitting a power series expansion) to the setting of regular functions defined on axially symmetric s-domains (on which a power series expansion is not guaranteed at all points).

\begin{theorem}[Extension Theorem]\label{circular extendability}
Let $\Omega \subseteq \mathbb{H}$ be an s-domain, and let $f: \Omega \to \mathbb{H}$ be a regular function. There exists a unique regular extension $\widetilde f: \widetilde \Omega \to \mathbb{H}$ of $f$ to the symmetric completion  $\widetilde \Omega$ of $\Omega$.
\end{theorem}
\begin{proof}
By construction the domain $\widetilde \Omega$ is axially symmetric and, by proposition \ref{circolarizzazionedidomini},  it is an s-domain. Moreover, the s-domain $\Omega \subseteq \widetilde \Omega$ contains an axially symmetric s-domain $D$ obtained as follows: for any $x\in \Omega\cap \mathbb{R}$ consider $\varepsilon_x>0$ such that $B(x, \varepsilon_x)\subseteq \Omega$ and define
$$
D=\bigcup_{x\in \Omega\cap \mathbb{R}}B(x, \varepsilon_x).
$$
Consider the largest axially symmetric s-domain $M$, with $D\subseteq M\subseteq \widetilde \Omega$, on which $f_{|D}$ can be extended to a regular function $\widetilde f: M \to \mathbb{H}$ and notice that, being $M$ an s-domain, the extension $\widetilde f$ is unique. Suppose $M \subsetneq \widetilde \Omega$. If this is the case, there exists  a point $p=u+vL\in \widetilde \Omega\cap \partial M$. Since $p=u+vL\in \widetilde \Omega$, we can find $J\in \mathbb{S}$ such that $u+vJ\in \Omega$. Since $\Omega$ is open, there exists $\delta > 0$ such that $B(u+vJ, \delta)\subseteq \Omega$. Therefore there is $K\in \mathbb{S}$ ($K\neq J$) close enough to $J$, and $\varepsilon >0$, so that the two discs $\Delta_J(u+vJ, \varepsilon)=\{x+yJ : (x-u)^2+(y-v)^2<\varepsilon \}$ and $\Delta_K(u+vK, \varepsilon)=\{x+yK : (x-u)^2+(y-v)^2<\varepsilon \}$ are both contained in $\Omega$. Let us define
\begin{equation}\label{effetilde}
\widetilde g (x+yI) =(J-K)^{-1}\Big[  J f(x+yJ)-Kf(x+yK)\Big]
+I(J-K)^{-1}\Big[ f(x+yJ)-f(x+yK)]\Big]
\end{equation}
on the symmetric completion $\widetilde \Delta_L(u+vL, \varepsilon)$ of $\Delta_L(u+vL, \varepsilon)$, i.e. on the set $\{x+yI : (x-u)^2+(y-v)^2<\varepsilon, I\in \mathbb{S}\}$. Theorem \ref{GRF} implies now that $\widetilde f \equiv \widetilde g$ in the circular open set $\widetilde \Delta_L(u+vL, \varepsilon)\cap M$, which is non empty due to $p=u+vL\in \partial M$. Since $\widetilde g$ is regular on $\widetilde \Delta_L(u+vL, \varepsilon)$, the function defined as
\begin{equation}
\widetilde h(q)=\left\{\begin{array}{ll}
\widetilde f(q) \textnormal{ for } q \in M\\
\widetilde g(q) \textnormal{ for } q \in \widetilde \Delta_L(u+vL, \varepsilon)
\end{array}\right.
\end{equation}
is the (unique) regular extension of $\widetilde f$ to the axially symmetric s-domain $M \cup \widetilde \Delta_L(u+vL, \varepsilon)$, which strictly contains $M$. By definition of $M$, the conclusion is that $M= \widetilde \Omega$, and $\tilde f\equiv f$ on $\Omega$ and the proof is complete.
\end{proof}

\begin{theorem}[General Extension Formula] \label{gef} For $J\in \mathbb{S}$, let $\Omega_J\subseteq L_J$ be a domain of $L_J$ such that $\Omega_J\cap \mathbb{R}\neq \emptyset$. For $K\in \mathbb{S}$ with $J\neq K$, define $\Omega_K = \{x+yK : x+yJ\in \Omega_J\}$. Let $r: \Omega_J \to \mathbb{H}$ and $s: \Omega_K \to \mathbb{H}$ be such that
\begin{eqnarray*}
(\frac{\partial }{\partial x} + J\frac{\partial }{\partial y})r(x+yJ)\equiv 0\\
(\frac{\partial }{\partial x} + K\frac{\partial }{\partial y})s(x+yK)\equiv 0\\
\end{eqnarray*}
and that $r(x)=s(x)$ for all $x\in \Omega_J\cap \mathbb{R}=\Omega_K\cap \mathbb{R}$. Then
there exists a unique regular function $\widetilde f: \widetilde \Omega_J \to \mathbb{H}$ that extends $r$ and $s$ to the symmetric completion $\widetilde\Omega_J $ of $\Omega_J$. Such $\widetilde f$ is defined by
\begin{equation}\label{ext}
\widetilde f(x+yI)
=(J-K)^{-1}[  J r(x+yJ)-Ks(x+yK)]
+I(J-K)^{-1}[ r(x+yJ)-s(x+yK)]
\end{equation}
for all $x+yI\in  \widetilde\Omega_J$.
\end{theorem}
\begin{proof}
Since $\Omega_J\cap \mathbb{R}\neq \emptyset$ and $\Omega_J$ is a domain in $L_J$, the domain $\widetilde\Omega_J $ is an axially symmetric s-domain of $\mathbb{H}$. Since, by assumption,  $\frac{\partial}{\partial x}r(x+yJ)= -J\frac{\partial}{\partial y}r(x+yJ)$ and $\frac{\partial}{\partial x}s(x+yK)= -K\frac{\partial}{\partial y}s(x+yK)$, it is easy to see that the same arguments used in the proof of the first part of the statement of theorem \ref{GRF} can be used in this case and the function defined in (\ref{ext}) is regular.
\end{proof}

\begin{remark} {\rm In the last statement, the hypothesis that $\Omega_J\cap \mathbb{R}\neq \emptyset $ cannot be removed. Indeed, for $J\in \mathbb{S}$, the domain $\Omega_J=\{ (x+yJ) : x^2+(y-2)^2 <1\}$ of $L_J$ is such that $\widetilde\Omega_J $ is not an s-domain, and, as we already know, there is no hope to have a unique regular extension of $r$ and $s$ to $\widetilde \Omega$. }

\end{remark}

\begin{lemma}[Extension Lemma]\label{Extension Lemma}
Let $J\in\mathbb{S}$ and let $D$ be
a domain in $L_J$, symmetric with respect to the real axis and
such that $D\cap\mathbb{R}\not=\emptyset$. Let $\Omega_D$ be the
axially symmetric s-domain defined by
$$
\Omega_D=\bigcup_{x+yJ\in D}(x+y\mathbb{S}).
$$
If $f:\ D\to L_J$ is holomorphic, then the function $\tilde f:\
\Omega_D\to\mathbb{H}$ defined by
\begin{equation}\label{distribution_two}
\tilde f(x+yI) =\frac{1}{2}\Big[   f(x+yJ)+f(x-yJ)\Big]
+I\frac{1}{2}\Big[ J[f(x-yJ)-f(x+yJ)]\Big]
\end{equation}
is the unique regular, infinitely differentiable extension of
$f$ to $\Omega_D$.
Similarly, if the functions $F, G:\ D\to L_J$ are holomorphic, if $K\in \mathbb{S}$ is such that $K\perp J$ and if the function $f:\ D\to \mathbb{H}$ is defined by $f=F+GK$, then the function $\tilde f$ obtained by extending $F$ and $G$ is the unique regular,
infinitely differentiable extension of
$f$ to $\Omega_D$.
This unique extension of $f$ will be denoted by
${\rm ext}(f)$.
\end{lemma}
\begin{proof}
The regularity of $\tilde f$ follows directly by the proof of
Theorem \ref{formula}. When $I=J$ we have that $\tilde f(x+yJ)=
f(x+yJ)$, and hence $\tilde f$ is the unique extension of $f$ by
the Identity Principle. Moreover, $\tilde f$ is infinitely
differentiable on the
axially symmetric s-domain $\Omega_D$. The second part of the statement has a completely analogous proof. \end{proof}

\section{Regular functions on axially symmetric slice domains}

We will show in this section that some of the basic results obtained in \cite{advances} are valid in the more general case in which the domains of definition of regular functions are axially symmetric s-domains. To this purpose, we will extend to such domains the $*$-product introduced in \cite{caterina} in the case of regular power series converging on balls centered at real points. We begin with the following:
\begin{proposition}\label{zeriolo} Let
$\Omega\subseteq \mathbb{H}$  be an axially symmetric s-domain and let $f :
\Omega \to \mathbb{H}$ be a regular function. Suppose that
there exists a imaginary unit $J\in $ $\mathbb{S}$ such that
$f(L_J)\subset L_J$. If there exists a imaginary unit $I\in
\mathbb{S}$ such that $I \notin L_{J}$ and that
$f(x_{0}+y_{0}I)=0$, then $f(x_{0}+y_{0}L)=0$ for all $L\in
\mathbb{S}$.
\end{proposition}
\begin{proof}
The only case we have to consider is $y_0\not= 0$. By formula
(\ref{formula}) we have on $L_{I}\cap\Omega$
$$
f(x+yI) =\frac{1}{2}\Big[   f(x+yJ)+f(x-yJ)\Big]
+I\frac{1}{2}\Big[ J[f(x-yJ)-f(x+yJ)]\Big],
$$
in particular $f(x_0+y_{0}I)=0$ implies that
$$
0 =\frac{1}{2}\Big[   f(x_0+y_0J)+f(x_0-y_0J)\Big]
+I\frac{1}{2}\Big[ J[f(x_0-y_0J)-f(x_0+y_0J)]\Big],
$$
and since $f(L_J)\subset L_J$ and $1$ and $I$ are linearly
independent on $L_J$ we get
$$
f(x_0+y_0J)=-f(x_0-y_0J)\qquad\qquad f(x_0-y_0J)=f(x_0+y_0J)
$$
so $f(x_0-y_0J)=f(x_0+y_0J)=0$. The conclusion follows
from the Representation Formula. \end{proof}


This last result has a consequence on the zero set of
holomorphic functions defined on domains intersecting the real
axis and symmetric with respect to it. Since,  by the Extension Lemma
\ref{Extension Lemma}, any holomorphic
function $f$ can be uniquely extended  to a regular function defined over quaternions,
one may ask which zeros of $f$ will remain
isolated after the extension, and which will become ``spherical''.
The answer to this natural question is given in the following Proposition, whose proof is an immediate consequence of the Representation Formula \ref{formula} combined with the Extension Lemma \ref{Extension Lemma}.
\begin{proposition}
Let $J\in\mathbb{S}$ and let $\Omega$ be a domain in $L_J$, symmetric
with respect to the real axis and such that
$\Omega\cap\mathbb{R}\not=\emptyset$. Let $\widetilde{\Omega}$ be its symmetric
completion.
Let $f:\ \Omega\to L_J$ be a holomorphic function and $\tilde f:\
\widetilde{\Omega}\to\mathbb{H}$ be its regular extension. If $\tilde
f(x_{0}+y_{0}J)=0$, for $y_0\not= 0$, then the zero
$(x_{0}+y_{0}J)$ of $\tilde f$ is not isolated, or equivalently
$\tilde f(x_{0}+y_{0}L)=0$ for all $L\in \mathbb{S}$, if, and only
if,
\begin{equation}\label{condition}
f(x_{0}+y_{0}J)=f(x_{0}-y_{0}J)=0.
\end{equation}
\end{proposition}
\begin{proof}
If the condition (\ref{condition}) is satisfied, then by the Representation Formula, the value of $\tilde f(x_{0}+y_{0}L)$ is zero for all $L\in \mathbb{S}$. The converse of the statement is immediate.
\end{proof}

We now  turn our attention to the problem of defining a suitable product of two regular functions.
As it is well known, the product of two regular functions is
not, in general, regular. In the case of regular polynomials,
i.e. polynomials with quaternionic coefficients on the right, one
can exploit the standard multiplication of polynomials in a skew
field (see e.g. \cite{lam}). This
product extends naturally to a product of regular power series on their domain of convergence, see \cite{caterina}, as
\begin{equation}\label{*perserie}
\Big(\sum_{n\geq 0} q^n a_n\Big) * \Big(\sum_{n\geq 0} q^n
b_n\Big)=\sum_{n\geq 0} q^n  \Big(\sum_{r=0}^na_rb_{n-r}\Big),\qquad(a_r,b_r\in\mathbb{H}).
\end{equation}
Our next goal is to generalize the regular product in the case of regular functions defined on axially symmetric s-domains. Let $\Omega\subseteq\mathbb{H}$ be such
a domain and let
 $f,g:\ \Omega\to\mathbb{H}$ be regular functions. For any $I,J\in\mathbb{S}$, with
 $I\perp J$, the Splitting Lemma guarantees the existence of four holomorphic functions
 $F,G,H,K: \ L_I\cap\Omega \to L_I$ such
 that for all $z=x+yI\in  L_I\cap\Omega$
 $$
 f_I(z)=F(z)+G(z)J \qquad g_I(z)=H(z)+K(z)J.
 $$
We define the function $f_I*g_I:\ L_I\cap\Omega \to L_I$ as
\begin{equation}\label{f*g}
f_I*g_I(z)=[F(z)H(z)-G(z)\overline{K(\bar
z)}]+[F(z)K(z)+G(z)\overline{H(\bar z)}]J.
\end{equation}
Then  $f_I*g_I(z)$ is obviously holomorphic and hence its unique
regular extension to $\Omega$ defined, according to the
Extension Lemma \ref{Extension Lemma},
 by
$$
f*g(q)={\rm ext}(f_I*g_I)(q),
$$
 is regular on $\Omega$.
 \begin{definition}
Let $\Omega\subseteq\mathbb{H}$ be an axially symmetric s-domain and let
 $f,g:\ \Omega\to\mathbb{H}$  be regular. The function
$$f*g(q)={\rm ext}(f_I*g_I)(q)$$ defined as the extension of (\ref{f*g})
is called the regular product of $f$ and $g$. This product is
called $*$-product or regular product.
 \end{definition}
\begin{remark}{\rm
It is immediate to verify that the $*$-product is associative,
distributive but, in general, not commutative.}
\end{remark}
\begin{remark}\label{J*H}{\rm
Let $H(z)$ be a holomorphic function in the variable $z\in L_I$
and let $J\in\mathbb{S}$ be an imaginary unit orthogonal to $I$. Then by the
definition of $*$-product we obtain $J*H(z)=\overline{H(\bar
z)}J$.}
\end{remark}
As it can be easily verified, given a regular function, its reciprocal with respect to the standard product is not in general regular, indeed the function $f(q)=q-j$ is regular but its reciprocal $f(q)^{-1}=(x-j)^{-1}$ is not.
One may wonder if it is possible to introduce a notion of regular reciprocal using the $*$-product. The answer, as we will see,  is positive and requires a couple of additional notions which are of independent interest. The idea underlying the notion of regular reciprocal is based on the standard idea to construct the reciprocal of a quaternion. Given $q\in\hh$, its reciprocal is given by $q^{-1}=\bar{q}/|q|^2$. The crucial fact is that the denominator is a real number which is  obtained as $\bar{q}q=q\bar{q}$. By defining a suitable
notion of conjugate of a function we will show how to find the regular reciprocal of a regular function.
Let $\Omega\subseteq\mathbb{H}$ be an axially symmetric s-domain and let
 $f:\ \Omega\to\mathbb{H}$ be a regular function. For any $I,J\in\mathbb{S}$, with
 $I\perp J$, the Splitting Lemma guarantees the existence of two holomorphic functions
 $F,G: \ L_I\cap\Omega \to L_I$ such
 that for all $z=x+yI\in  L_I\cap\Omega$
 $$
 f_I(z)=F(z)+G(z)J.
 $$
Let us define the function $f_I^c:\ L_I\cap\Omega \to L_I$ as
\begin{equation}\label{f^c}
f_I^c(z)=\overline{F(\bar z)}-G(z)J.
\end{equation}
Then  $f_I^c(z)$ is obviously holomorphic and hence its unique
regular extension to $\Omega$ defined, according to the
Extension Lemma \ref{Extension Lemma},
 by
$$
f^c(q)={\rm ext}(f_I^c)(q),
$$
 is regular on $\Omega$.
 \begin{definition}
Let $\Omega\subseteq\mathbb{H}$ be an axially symmetric s-domain and let
 $f:\ \Omega\to\mathbb{H}$  be regular. The function
$$f^c(q)={\rm ext}(f_I^c)(q)$$ defined by the extension (\ref{f^c}) is called the regular conjugate of $f$.
 \end{definition}
Since any axially symmetric s-domain contains a ball with center at a real point, the notion of regular conjugate given in \cite{caterina} by means of the power series expansion of a regular function $f$ can be used at least locally. In fact these two notions coincide, as it is proved in the next result:
\begin{proposition}
Let $f:\ B(p_0,R)\to\mathbb{H}$ be a regular function on an
open ball in $\mathbb{H}$ centered at a real point $p_0$. If
$f(q)=\sum_{n\geq 0} q^n a_n$ then $f^c(q)=\sum_{n\geq 0} q^n
\bar{a}_n$.
\end{proposition}
\begin{proof} We will suppose without loss of generality that $p_0=0$.
By Proposition \ref{serie_domini}, given any $I\in\mathbb{S}$, the
coefficients of the power series expansion of $f$ can be obtained as
the coefficients of the power series of $f_I$. For all
$z=x+yI\in L_I\cap B(0,R)$ and $J\in\mathbb{S}$ with  $J\perp I$
we get
$$
f_I(z)=F(z)+G(z)J=\sum_{n\geq 0}z^n
\frac{1}{n!}\Big(\frac{\partial F}{\partial x}(0) +\frac{\partial
G}{\partial x}(0)J\Big)=\sum_{n\geq 0}z^n \frac{1}{n!}\partial_s
f(0)
$$
and hence
$$
f^c_I(z)=\overline{F(\bar z)}-G(z)J=\sum_{n\geq 0}z^n \frac{1}{n!}
\Big(\overline{\frac{\partial F}{\partial x}}(0) -\frac{\partial
G}{\partial x}(0)J\Big)=\sum_{n\geq 0}z^n
\frac{1}{n!}\overline{\partial_s f(0)}
$$
which proves the statement.
\end{proof}
\begin{remark}{\rm
When dealing with regular power series defined on open balls in
$\mathbb{H}$ (centered at real points), the natural definition of
regular conjugate is the one suggested in the Proposition above
(see \cite{caterina}).}
\end{remark}
Using the notion of $*$-multiplication of regular functions, it
is possible to associate to any regular function $f$ its
``symmetrization'', denoted by $f^s$. The symmetrization of a regular function will be used in the sequel to introduce the announced regular reciprocal of a function $f$.

Let $\Omega\subseteq\mathbb{H}$ be an axially symmetric s-domain and let
 $f:\ \Omega\to\mathbb{H}$ be a regular function. Using, as before, the Splitting Lemma we can write
 $$
 f_I(z)=F(z)+G(z)J,
 $$
 with
$F,G: \ L_I\cap\Omega \to L_I$  holomorphic functions.
We define the function $f^s:\ L_I\cap\Omega \to L_I$ as
\begin{equation}\label{f^s}
f^s_I=f_I*f^c_I=(F(z)+G(z)J)*(\overline{F(\bar z)}-G(z)J)
\end{equation}
$$
=[F(z)\overline{F(\bar z)}+G(z)\overline{G(\bar
z)}]+[-F(z)G(z)+G(z)F(z)]J
$$
$$
=F(z)\overline{F(\bar z)}+G(z)\overline{G(\bar z)}=f^c_I*f_I.
$$
Then  $f_I^s$ is obviously holomorphic and hence its unique
regular extension to $\Omega$ defined by
$$
f^s(q)={\rm ext}(f^s_I)(q)
$$
is regular.
\begin{definition}
Let $\Omega\subseteq\mathbb{H}$ be an axially symmetric s-domain and let
 $f:\ \Omega\to\mathbb{H}$  be regular. The function
$$f^s(q)={\rm ext}(f_I^s)(q)$$ defined by the extension of (\ref{f^s}) is called the symmetrization of $f$.
\end{definition}
\begin{remark}\label{symmLI}
Notice that formula (\ref{f^s}) yields that, for all
$I\in\mathbb{S}$, $f^s(L_I\cap \Omega)\subseteq L_I$.
\end{remark}
 The symmetrization process well behaves with respect to the $*$-product and the conjugation:
\begin{proposition}\label{propmoltiplicative}
Let $\Omega\subseteq\mathbb{H}$ be an axially symmetric s-domain and let
 $f,g:\ \Omega\to\mathbb{H}$ be regular functions.
Then $(f*g)^c = g^c*f^c$ and
\begin{equation}
(f*g)^s = f^s g^s = g^s f^s.
\end{equation}
\end{proposition}
\begin{proof}
It is sufficient to show that $(f*g)^c = g^c*f^c$. As usual,
we can use the Splitting Lemma to write on $L_I\cap\Omega$ that
$f_I(z)=F(z)+G(z)J$ and $g_I(z)=H(z)+K(z)J$. We get
$$
f_I*g_I(z)=[F(z)H(z)-G(z)\overline{K(\bar
z)}]+[F(z)K(z)+G(z)\overline{H(\bar z)}]J
$$
and hence
$$
(f_I*g_I)^c(z)=[\overline{F(\bar z)}\ \overline{H(\bar
z)}-\overline{G(\bar z)}{K(z)}] -[F(z)K(z)+G(z)\overline{H(\bar
z)}]J.
$$
We now compute
$$
g_I^c(z)*f_I^c(z)=(\overline{H(\bar z)}-K(z)J)*(\overline{F(\bar
z)}-G(z)J)
$$
$$
=\overline{H(\bar z)}*\overline{F(\bar z)}-\overline{H(\bar
z)}*G(z)J -K(z)J*\overline{F(\bar z)}+K(z)J*G(z)J
$$
and conclude by Remark \ref{J*H}. \end{proof}

An important fact which generalizes the analogue result in \cite{caterina} is the following:
\begin{proposition}
Let $\Omega\subseteq\mathbb{H}$ be an axially symmetric s-domain and let
$f,g:\ \Omega\to\mathbb{H}$ be regular functions. Then
\begin{equation}\label{formula prodotto}
f*g(q) = f(q)\ g(f(q)^{-1} q f(q)),
\end{equation}
for all $q\in\Omega$.
\end{proposition}
\begin{proof}
Let $I$ be any element in $\mathbb{S}$ and let $q=x+yI$. If
$f(x+yI)\not=0$, it is easy to verify that
$$
f(x+yI)^{-1}(x+yI)f(x+yI)=x+yf(x+yI)^{-1}If(x+yI)
$$
with $f(x+yI)^{-1}If(x+yI)\in\mathbb{S}$. Using now the
representation formula (\ref{formula}) for the function $g$, we
get
$$
g(f(q)^{-1} q f(q))=g(x+yf(x+yI)^{-1}If(x+yI))
$$
$$
=\frac{1}{2}\{g(x+yI)+g(x-yI)-f(x+yI)^{-1}If(x+yI)[Ig(x+yI)-Ig(x-yI)
]\}
$$
and
$$
\psi(q):=f(q)g(f(q)^{-1} q f(q))
$$
$$
=\frac{1}{2}\{f(x+yI)[g(x+yI)+g(x-yI)]-If(x+yI)[Ig(x+yI)-Ig(x-yI)
]\}.
$$
If we prove that the function $f(q)g(f(q)^{-1} q f(q))$ is
regular, then our assertion will follow by the Identity
Principle, since formula (\ref{formula prodotto}) holds on a small
open ball of $\Omega$ centered at a real point (see Proposition
\ref{serie_domini}). Let us compute (with obvious notation for the
derivatives):
$$
\frac{\partial}{\partial x}\psi(x+yI)
$$
$$
= \frac{1}{2}\{f_x(x+yI)[g(x+yI)+g(x-yI)]
-If_x(x+yI)[Ig(x+yI)-Ig(x-yI) ]\}
$$
$$
+\frac{1}{2}\{f(x+yI)[g_x(x+yI)+g_x(x-yI)]
-If(x+yI)[Ig_x(x+yI)-Ig_x(x-yI) ]\}
$$
and
$$
I\frac{\partial}{\partial y}\psi(x+yI)
$$
$$
= \frac{1}{2}\{If_y(x+yI)[g(x+yI)+g(x-yI)]
+f_y(x+yI)[Ig(x+yI)-Ig(x-yI) ]\}
$$
$$
+\frac{1}{2}\{If(x+yI)[g_y(x+yI)+g_y(x-yI)]
+f(x+yI)[Ig_y(x+yI)-Ig_y(x-yI) ]\}.
$$
By using the three relations
$$
f_x(x+yI)+If_y(x+yI)= g_x(x+yI)+Ig_y(x+yI)=
g_x(x-yI)-Ig_y(x-yI)=0,
$$
we obtain that $(\frac{\partial}{\partial
x}+I\frac{\partial}{\partial y})\psi(x+yI)=0$. The fact that $I$
is arbitrary proves the assertion. \end{proof}

As a consequence, we can now identify the zeros of the $*$-product of two regular functions $f$ and $g$ defined on an axially symmetric s-domain, in terms of the zeros of $f$ and of $g$.

\begin{theorem}\label{zeriprodotto}
Let $\Omega\subseteq\mathbb{H}$ be an axially symmetric s-domain and let
$f,g:\ \Omega\to\mathbb{H}$ be regular functions. Then $f*g(q) =
0$ if and only if $f(q) = 0$ or $f(q) \neq 0$ and $g(f(q)^{-1} q
f(q))=0$.
\end{theorem}

\begin{proof}
Formula (\ref{formula prodotto}) implies
$$f*g(q) =
f(q)\ g(f(q)^{-1} q f(q)).$$ Therefore $f*g(q)=0$ if and only if
$f(q)g(f(q)^{-1} q f(q))=0$ if and only if either
 $f(q)=0$ or $f(q) \neq 0$ but then  $g(f(q)^{-1} q f(q))=0$.
\end{proof}

In particular, if $f*g$ has a zero in $S=x+y\s$ then either $f$ or
$g$ have a zero in $S$. However, the zeros of $g$ in $S$ need not
be in one-to-one correspondence with the zeros of $f*g$ in $S$
which are not zeros of $f$.

Finally, we use the notion of regular conjugate and symmetrization introduced above to define the
regular reciprocal of a, regular function:
\begin{definition}
Let $\Omega\subseteq\mathbb{H}$ be an axially symmetric s-domain, let
 $f:\ \Omega\to\mathbb{H}$ be a regular function and $Z_{f^s}$ be the set of zeros of its symmetrization $f^s$. We define the regular reciprocal of $f$ as the function
 $f^{-*}: \Omega\setminus Z_{f^s}\to\hh$ given by
\begin{equation}\label{inverse}
f(q)^{-*}=f^s(q)^{-1}f^c(q).
\end{equation}
\end{definition}

The notion of regular reciprocal is quite natural in the environment of regular functions and has already been implicitly used while establishing a new Cauchy formula, as we point out in the following

\begin{remark}{\rm The use of the Cauchy kernel which appears in \cite{cgs} to prove the validity of a Cauchy formula on axially symmetric s-domains was suggested by the functional calculus introduced in \cite{cgssann}, \cite{cgss}, \cite{cosa} and \cite{functionalcss}. This Cauchy kernel can be also found as the regular reciprocal of the function $S(q)= q-s$ where $s$ is any fixed quaternion. Indeed we have that
$S^{-*}(q)=(q^2-2{\rm Re}[s]q+|s|^2)^{-1}(q-\bar{s})$. }
\end{remark}



\begin{thebibliography}{99}
\bibitem{cgs} F. Colombo, G. Gentili, I. Sabadini, {\em A Cauchy kernel for slice regular functions},  arXiv:0902.4771  [math.CV].

\bibitem{cgssann} F. Colombo, G. Gentili, I. Sabadini, D.C. Struppa,
{\em A functional calculus in a  non commutative  setting},
 Electron. Res. Announc. Math. Sci, {\bf 14} (2007), 60-68.

\bibitem{cgss} F. Colombo, G. Gentili, I. Sabadini, D.C. Struppa, {\em
Non commutative functional calculus:  bounded operators}, to appear in: Compl. Anal. Oper. Theory.

\bibitem{cgss2} F. Colombo, G. Gentili, I. Sabadini, D.C. Struppa, {\em An overview on functional calculus in different settings},
{\em  Hypercomplex analysis}, Trends in Math., Birkh\"auser, Basel, 2009, 69--99.




\bibitem{cosa} F. Colombo, I. Sabadini, {\em On some properties of the quaternionic
functional calculus}, to appear in: J. Geom. Anal.



\bibitem{csss} F. Colombo, I. Sabadini, F. Sommen, D.C.
Struppa, {\em Analysis of Dirac Systems and Computational Algebra},
Progress in Mathematical Physics, Vol. 39, {Birkh\"auser}, Boston,
2004.


\bibitem{functionalcss} F. Colombo, I. Sabadini, D.C. Struppa, {\em
 A new functional calculus for noncommuting operators},
 J. Funct. Anal., {\bf 254} (2008), 2255--2274.

\bibitem{slicecss} F. Colombo, I. Sabadini, D.C. Struppa, {\em
 Slice monogenic functions}, to appear in: Israel Journal of Mathematics.




\bibitem{fueter 1} R. Fueter, {\em Die Funktionentheorie der
Differentialgleichungen $\bigtriangleup u = 0$ und $\bigtriangleup \bigtriangleup u = 0$
mit vier reellen Variablen}, Comm. Math. Helv. {\bf 7} (1934),
307--330.

\bibitem{fueter 2} R. Fueter, {\em \"{U}ber eine Hartogsschen Satz},
Comment. Math. Helv., {\bf 12} (1939/40), 75--80.

\bibitem{caterina} G. Gentili, C. Stoppato, {\em Zeros of regular functions and polynomials of
a quaternionic variable}, Michigan Math. J., {\bf 56} (2008), 655--667.

\bibitem{open} G. Gentili, C. Stoppato, {\em The open mapping
theorem for quaternionic regular functions}, E-print
arXiv:0802.3861v1 [math.CV], to appear in: Ann. Sc. Norm. Super. Pisa Cl. Sci. (5).

\bibitem{gs} G. Gentili, D.C. Struppa, {\em A new approach to Cullen-regular functions
of a quaternionic variable},
 C.R. Acad. Sci. Paris, {\bf 342} (2006), 741--744.

\bibitem{advances} G. Gentili, D.C. Struppa, {\em A new theory of regular functions
of a quaternionic variable}, Adv. Math. {\bf 216} (2007),  279--301.


\bibitem{lam} T.Y. Lam, \textit{A first course in noncommutative
rings}. Graduate Texts in Mathematics, 123. Springer-Verlag, New
York, 1991, 261--263.
\end{thebibliography}
\end{document}